\documentclass[12pt]{article}

\nonstopmode
\usepackage{graphicx,xspace,colortbl}

\usepackage{amsmath,amsthm,amsfonts}
\usepackage{fancybox}

\nonstopmode

\newtheorem{theorem}{Theorem}
\newtheorem{lemma}[theorem]{Lemma}

\theoremstyle{definition}
\newtheorem{definition}{Definition}

\newtheorem{remark}[definition]{Remark}
\newtheorem{remarks}[definition]{Remarks}

\def\CF{{\cal F}}

\def\CL{{\cal L}}

\def\qed{\hfill$\sqcap\kern-8.0pt\hbox{$\sqcup$}$\\}
\def\beq{\begin{eqnarray}}
\def\eeq{\end{eqnarray}}
\def\beqq{\begin{eqnarray*}}
\def\eeqq{\end{eqnarray*}}

\def\be{\begin{equation}}
\def\ee{\end{equation}}

\newcommand{{\X}}{{\bf X}}
\newcommand{{\x}}{{\bf x}}
\newcommand{{\Z}}{{\bf Z}}
\newcommand{{\z}}{{\bf z}}
\newcommand{{\Y}}{{\bf Y}}
\newcommand{{\y}}{{\bf y}}
\newcommand{{\F}}{{\bf F}}
\newcommand{{\bbeta}}{{\bf \beta}}
\newcommand{{\bsigma}}{{\bf \sigma}}
\newcommand{{\bL}}{{\bf L}}
\newcommand{{\bW}}{{\bf W}}

\title{   On the first passage time for  Brownian motion subordinated by a Levy process 
}
\author{T. R. Hurd$^{\;1\; *}$ and A. Kuznetsov$\;{}^2$ \thanks{{Research supported by the
Natural Sciences and Engineering Research Council of Canada and
MITACS Mathematics of Information Technology and Complex Systems,
Canada}} \\ ${}^1$\;Dept. of Mathematics and Statistics\\ McMaster
University\\Hamilton ON L8S 4K1\\Canada
 \\ \\
${}^2$\;Dept. of Mathematical Sciences\\  University of New
Brunswick
\\Saint John NB E2L 4L5 \\  Canada \\
 }\date{May 21, 2008}

\begin{document}
\maketitle

\begin{abstract}
This paper considers the class of L\'evy processes that can be written as a Brownian motion time changed by an independent L\'evy subordinator. Examples in this class include the variance gamma model, the normal inverse Gaussian model, and other processes popular in financial modeling. The question addressed is the precise relation between the standard first passage time and an alternative notion, which we call first passage of the second kind, as suggested by \cite{Hurd07a} and others. We are able to prove that standard first passage time is the almost sure limit of iterations of first passage of the second kind. Many different problems arising in financial mathematics are posed as first passage problems, and motivated by this fact, we are lead to consider the implications of the approximation scheme for fast numerical methods for computing first passage. We find that the generic form of the iteration can be competitive with other numerical techniques.  In the particular case of the VG model, the scheme can be further refined to give very fast algorithms.


 \end{abstract}

 \section{Introduction}
 First passage problems are a classic aspect of stochastic processes that arise in many areas of application. In mathematical finance, for example, first passage problems lie at the heart of such issues as credit risk modelling, pricing barrier options, and the optimal exercise of american options. If $X_t$ is any process with initial value $X_0=x_0$, the first passage time to a lower level $b$ is defined to be the stopping time
 \be t^*_b(x_0)=\inf\{t\ge 0|X_t\le b\}\ee
The distributional properties of $t^*$ are well studied when the underlying process $X$ is a diffusion, but when $X$ has jumps, for example, for L\'evy processes, much less is known. 
 
 Explicit formulas  for first passage are known only for certain special L\'evy processes, specifically the Kou-Wang jump diffusion model \cite{KouWang03} and its generalization to phase type processes \cite{AsmAvrPis04}. 
 In the setting of general L\'evy processes, one analytical approach to first passage is the Wiener-Hopf method, described in  \cite{AlilKypr05}. This method can for example compute first passage problems for processes with one-sided jumps.  A second general approach to first passage is to solve the Fokker-Planck equation for the probability density of $X_t$ conditioned on the set $\{ t^*>t\}$.  For L\'evy processes this amounts to solving a certain  linear partial integral differential equation (PIDE) with Dirichlet boundary conditions. The PIDE approach to first passage in the general setting has apparently not been widely implemented: it seems that when faced with difficult first passage problems, practitioners often fall back on Monte Carlo methods. 
 
  Our purpose here is to present a new approach to first passage problems applicable whenever the underlying L\'evy process can be realized as a L\'evy subordinated Brownian motion (LSBM), that is, whenever $X$ can be constructed as $\tilde W\circ T$ where $\tilde W$ is a standard drifting Brownian motion and $T$ is a non-decreasing L\'evy process independent of $\tilde W$. The class of L\'evy processes that are realizable as LSBMs  is broad enough to include most of the L\'evy processes that have so far been used in finance, such as the Kou-Wang model, the variance gamma (VG) model, and the normal inverse Gaussian (NIG) model. 
 
 The basis for our approach is that for processes that are realizable as time changes of Brownian motions, there is an alternative notion that is also relevant, namely, the first time the time change exceeds the first passage time of the Brownian motion. This notion, called {\em first passage of the second kind} in \cite{Hurd07a}, shares some characteristics with the usual first passage time and can be applied in a similar way. The usefulness of this new concept is that it can be computed efficiently in many cases where the usual first passage time cannot.  
 
In the present paper, we study first passage for LSBMs and show how the first passage of the second kind is the first of a sequence of stopping times that converges almost surely to the first passage time. Expressed differently, first passage can be viewed as a stochastic sum of first passage times of the second kind. This sequence leads to a convergent and computable expansion for the first passage probability distribution function $p^*$ in terms of a similar function $p^*_1$ that describes the first passage distribution of the second kind. The outline of the paper is as follows. In Section 2, we define the objects needed to understand first passage time, and prove the expansion formula for first passage. In Section 3, we demonstrate the usefulness of this expansion by proving several explicit two dimensional integral formulas for $p^*_1$, the first passage distribution of the second kind. Section 4 provides two proofs of the convergence of the expansion.    The first proof is a proof of convergence in distribution, the second is in the pathwise (almost sure) sense.  Section 5 focusses on the special case of the variance gamma (VG) model. In this important example, the formula for $p^*_1$ is reduced to a one-dimensional integral (involving the exponential integral function),. In Section 6,  the expansion  of the function $p^*$ is studied numerically, and found to be numerically stable and efficient.

 \section{First Passage for LSBMs}
 Let $X_t$ be a general L\'evy process with initial value $X_0=x_0$ and characteristics $(b,c,\nu)_h$ with respect to a truncation function $h(x)$ (see \cite{Applebaum04} or \cite{JacoShir87}). This means $X$ is an infinitely divisible process with identical independent increments and c\'adl\'ag paths almost surely.  $b,c\ge 0$ are real numbers and $\nu$ is a sigma finite measure on $\mathbb{R}\setminus 0$ that integrates the function $1\wedge x^2$. By the L\'evy-Khintchine formula, the log-characteristic function of $X_t$ is 
 \be\label{LKformula}
 \log E[e^{iuX_t}]=iub-cu^2/2+\int_{\mathbb{R}\setminus 0}\left(e^{iux}-1-xh(x)\right)\nu(dx).
 \ee
 In what follows we will find it convenient to focus on the Laplace exponent of $X$: \[ \psi_X(u):=\log E[e^{-uX_1}]\]
For simplicity of exposition, we specialize slightly by assuming that $\nu$ is continuous with respect to Lebesgue measure $\nu(dx)=\nu(x) dx$, and  integrates $1\wedge |x|$, allowing us to take $h(x)=0$. 
 In this setting, the Markov generator of the process $X_t$ applied to any sufficiently smooth  function $f(x)$ is 
 \be\label{Markovgenerator}
 [\CL f](x)=b\partial_xf+\frac{c}2\partial^2_{xx}f+\int_{\mathbb{R}\setminus 0}\left(f(x+y)-f(x)\right)\nu(y)dy
 \ee

 \begin{definition}
For any $b\in\mathbb{R}$, the random variable $t^*_b=t^*_b(x_0):=\inf\{t\ |\ X_t\le b\}$ is called {\em the first passage time for level $b$}. When $b=0$, we drop the subscript and  $t^*:=\inf\{t\ |\ X_t\le 0\}$ is called simply {\em the first passage time of $X$}. 
\end{definition}

\begin{remarks}\begin{enumerate}
  \item Since distributions of the increments of $X$ are invariant under time and state space shifts, we can reduce computations of $t^*_b(x_0)$ to computations of $t^*(x_0-b)$. 
  \item A general L\'evy process is a mixture of a continuous Brownian motion with drift and a pure jump  process, and $t^*$ is the minimum of a predictable stopping time (coming from the diffusive part) and a totally inaccessible stopping time (coming from the down jumps). Only if $\mbox{supp}(\nu)\subset\mathbb{R}_+$ is $t^*$ predictable.  If $c=0$ or if $\nu$ has an infinite activity of down jumps (i.e. if $\nu(\mathbb{R}_-)=\infty$), then $t^*$ is  totally inaccessible.
  \item When $X_{t^*}-X_{t^*-}\ne 0$, we say that $X$ jumps across $0$, and define the {\em overshoot} to be $X_{t^*}$. 
\end{enumerate}

\end{remarks}

The central object of study in this paper is the joint distribution of $t^*$ and the overshoot $X_{t^*}$, in particular the joint probability density function
\be\label{jointpdf2}
p^*(x_0;s,x_1)=E_{x_0}[\delta(t^*-s)\delta(X_{t^*}-x_1)]\ee 
The marginal density of $t^*$ is $p^*(x_0;s)=\int_{-\infty}^0\ p^*(x_0;s,x_1)\ dx_1$.

In the introduction, we noted that general results on first passage for L\'evy processes, in particular results on the functions $p^*$, are difficult to come by. For this reason, we now focus on the special class of L\'evy processes that can be expressed as a drifting Brownian motion subjected to a time change by an independent L\'evy subordinator. Such {\em L\'evy subordinated Brownian motions} (LSBM) have been studied in a general review by \cite{CherShir02} and more specifically in \cite{Hurd07a}. The  general LSBM is constructed as follows:
\begin{enumerate}
  \item For an initial value $x_0>0$ and drift $\beta$, let $\tilde W_T=x_0 +W_T+\beta T$ be a drifting BM;
  \item For a L\'evy characteristic triple $(b,0,\mu)$ with $b\ge 0$ and $\mbox{supp}(\mu)\subset \mathbb{R}^+$, let the time change process $T_t$ be the associated nondecreasing L\'evy process (a subordinator), taken to be independent of  $W$;
  \item The time changed process $X_t=\tilde W_{T_t}$ is defined to be a LSBM.
\end{enumerate}

So constructed, it is known that $X_t$ is itself a L\'evy process: \cite{CherShir02} provide a characterization of which L\'evy processes are LSBMs.  It was observed in \cite{Hurd07a} that for any LSBM $X_t$, one can define an alternative notion of first passage time, which we denote here by $\tilde t$. 
\begin{definition} For any LSBM $X_t=\tilde W_{T_t}$ we define the {\em first passage time of $\tilde W$}  to be   $T^*=T^*(x_0)=\inf \{T: x_0+W_T+\beta T\le 0 \}$. Note $T^*(x_0)=0$ when $x_0\le 0$. The {\em first passage time of the second kind} of $X_t$ is defined to  $\tilde t=\tilde t(x_0)=\inf\{ t: T_t\ge T^*(x_0) \}$
\end{definition}

This definition of $\tilde t$, and its relation to $t^*$,  is illustrated in Figure \ref{fig1}.

  \begin{figure}\label{fig1}
\begin{center}
\includegraphics[width=0.45\textwidth]{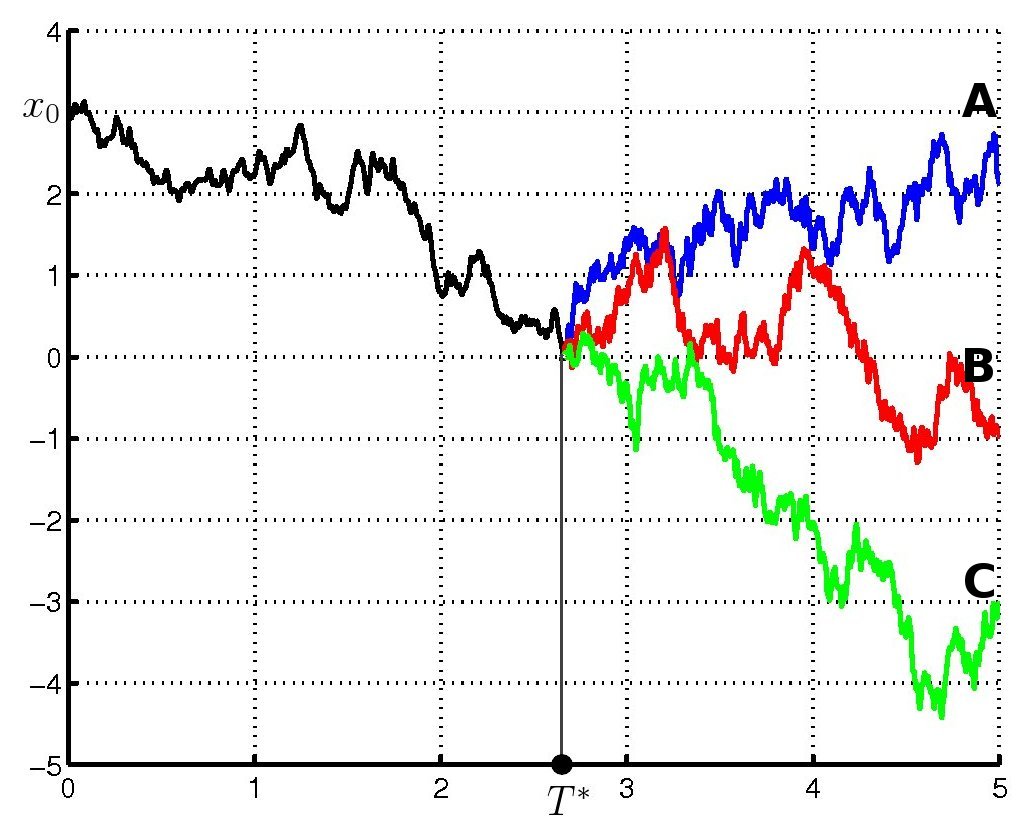}
\includegraphics[width=0.45\textwidth]{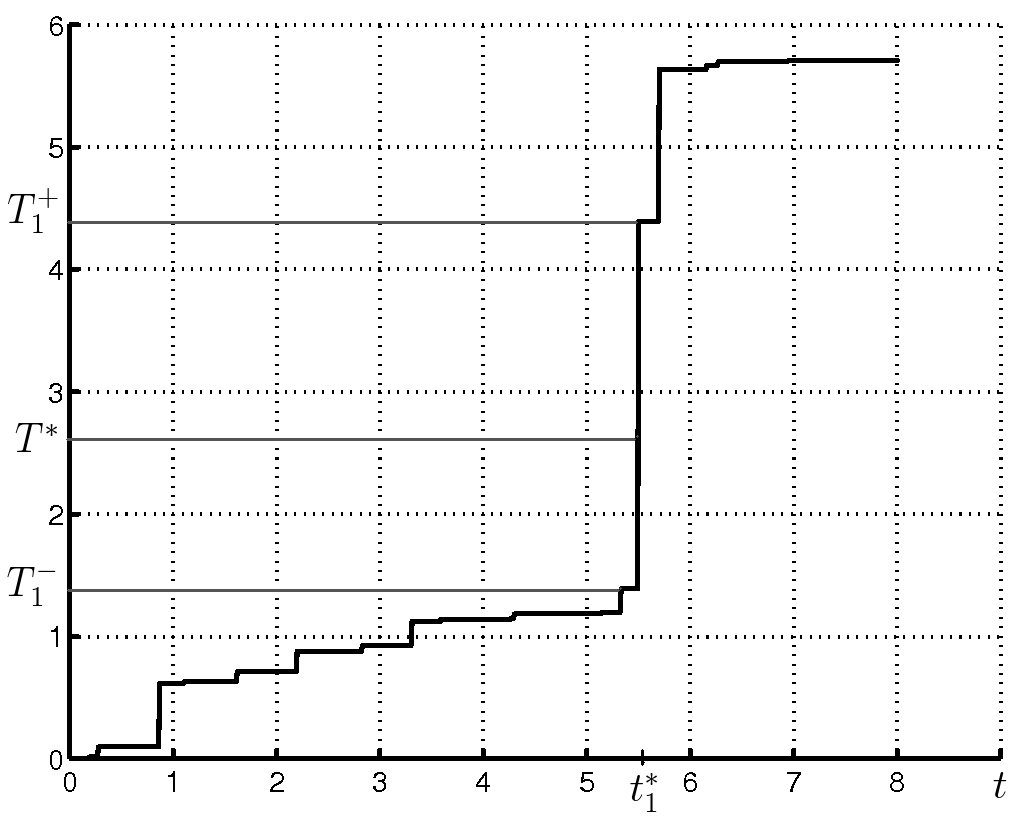}
\vspace{0.5cm} \caption{Three trajectories of $X_t$ and a sample
path of time change $\tau_s$}
\end{center}
\end{figure}

We now show that $\tilde t(x_0):=t_1^*(x_0)$ is the first of a sequence of approximations $\{t_i^*(x_0)\}_{i=1,2,\dots}$ to the stopping time $t^*$. The construction of $t^*_i(x_0,\pi)$ is pathwise: We introduce the second argument   $\pi\in\Omega$ which denotes a {\em sample path}, that is, a pair $(\omega,\tau)$ where $\omega$ is a continuous   drifting Brownian path $\tilde W$ and $\tau$ is a  c\'adl\'ag sample time change path $T$. Thus $\pi:(S,s)\to(\omega(S),\tau(s))_{S,s\ge 0}$. The natural ``big filtration'' $(\CF_t)_{t\ge 0}$ for time-changed Brownian motion has $\CF_t=\sigma\{\omega(S),\tau(s), S\le \tau(t),s\le t\}$. For any $t\ge 0$ there is a natural ``time translation'' operation on paths $\rho_t:(\omega,\tau)\to (\omega',\tau')$ where $\omega'(S)=\omega(S+\tau(t)), \tau'(s)=\tau(s+t)-\tau(t)$.

The construction of $\{t_i^*(x_0,\pi)\}_{i=1,2,\dots}$ for a given sample path $\pi$ is as follows. Inductively, for $i\ge 2$, we define the {\em time of the $i$-th excursion overjump}  
\be \label{tstar}
t^*_i(x_0,\pi)=\inf\{t\ge t^*_{i-1}(x_0,\pi):T_t-T_{t_{i-1}^*}\ge T^*(X_{t^*_{i-1}},\pi')\}
\ee where $\pi'=\rho_{t^*_{i-1}}(\pi)$ denotes a time shifted sample path. Note that $t^*_i(x_0,\pi)=t^*_{i-1}(x_0,\pi)$ if and only if $X_{t^*_{i-1}(x_0,\pi)}\le 0$ or $t^*_{i-1}(x_0,\pi)=\infty$. At any excursion overjump event $t^*_i$, the time interval which covers the event has left and right endpoints $T^-_i=T_{t^*_i-}$ and $T^+_i=T_{t^*_i}$. The joint distribution of $t^*_i(x_0), X_{t^*_i(x_0)}$ can be written 
\be\label{jointPDF3}
p^*_i(x_0;s,x)=E_{x_0}[\delta(t^*_i-s)\delta(X_{t^*_i}-x)]\ee 

The definition of this sequence of stopping times is summarized by the pathwise equation
\be\label{pathwise}
t^*_i(x_0,\pi)=t^*_{i-1}(x_0,\pi){\bf 1}_{\{X_{t^*_{i-1}}\le 0\}} +\left(t^*_1(X_{t^*_{i-1}},\pi')+ t^*_{i-1}(x_0,\pi)\right){\bf 1}_{\{X_{t^*_{i-1}}> 0\}}, \quad i\ge 2
\ee
where $\pi'=\rho_{t^*_{i-1}}(\pi)$.
The identical increments property of the LSBM implies the joint probability densities satisfy the recursive relation
\be \label{p_iteration}
 p_{i}^*(x_0;s,x)= p_1^*(x_0;s,x)I(x\le 0)+\int\limits_0^{\infty} dy \int\limits_{0}^s du \;
    p_{i-1}^*(x_0;u,y)p_{1}^*(y;s-u,x),\quad i\ge 2 
    \ee
    Similarly, the PDF of the first passage time $t^*$ satisfies the relation
    \be \label{pdf_iteration}
 p_{i}^*(x_0;s)= \int\limits_{-\infty}^0 p_1^*(x_0;s,x)dx+\int\limits_0^{\infty} dy \int\limits_{0}^s du \;
    p_{i-1}^*(x_0;u)p_{1}^*(y;s-u,x),\quad i\ge 2. 
    \ee
    \bigskip\noindent
    {\bf Examples of LSBMs:\  } 
We note here three classes of L\'evy processes that can be written as LSBMs and have been used extensively in financial modeling.   
    \begin{enumerate}
 \item The exponential model with parameters $(a,b,c)$ arises by taking $T_t$ to be the increasing process with drift $b\ge 0$  and jump measure $\mu(z)= c e^{-az}, c, a>0$ on $(0,\infty)$. The Laplace exponent of $T$ is
 \[ \psi_T(u):=-\log E[e^{-u T_1}]=bu+uc/(a+u).\]
 The resulting time-changed process $X_t:=\tilde W_{T_t}$ has triple  $(\beta b, b, \rho )$ with
 \[ \rho(y)=\frac{c}{\sqrt{\beta^2+2a}}e^{-(\sqrt{\beta^2+2a}+\beta)(y)^+-(\sqrt{\beta^2+2a}-\beta)(y)^-},\]
where  $(y)^+=\max(0,y), (y)^-=(-y)^+$. This forms a four dimensional subclass of the six-dimensional family of exponential jump diffusions studied by  \cite{KouWang03}. 
  \item The variance gamma (VG) model \cite{MadaSene90}  arises by taking $T_t$ to be a gamma process with drift defined by the characteristic triple $(b, 0, \mu)$ with $b\ge 0$ (usually $b$ is taken to be $0$) and jump measure $\mu(z)= (\nu z)^{-1} \exp(-z/\nu), \nu >0$ on $(0,\infty)$. The Laplace exponent of $T_t, t=1$ is
 \[ \psi_T(u):=-\log E[e^{-u T_1}]=bu+\frac{1}{\nu}\log(1+\nu u).
  \]

  The resulting time-changed process has triple  $(\beta b, b, \rho )$ with
 \[ \rho(y)=\frac{1}{\nu|y|}\exp\left(\beta x-\sqrt{\frac{2}{\nu}+\beta^2}|x|\right)\]
  \item The normal inverse Gaussian model (NIG) with parameters $\tilde\beta, \tilde \gamma$ \cite{Barn-Niel97} arises when $T_t$ is the first passage time for a second independent Brownian motion with drift $\tilde\beta>0$ to exceed the level $\gamma t$. Then 
 \[ \psi_T(u)=\tilde\gamma (\tilde\beta+\sqrt{\tilde\beta^2+2u})
 \]
 and  the resulting time-changed process has Laplace exponent
  \[ \psi_X(u)=x\mu+\tilde\gamma [\tilde \beta+\sqrt{\tilde\beta^2-u^{2}+2\tilde\beta u}].
 \]
\end{enumerate}

 \section{Computing First Passage of the Second Kind}

We have just seen that first passage for LSBMs admits an expansion as a  sum of first passage times of the second kind.  In this section, we show that this expansion can be useful, by proving  several equivalent integral formulas for computing the structure function $p_1^*(x_0;s,x_1)$ for general LSBMs. While the equivalence of these formulas can be demonstrated analytically, their numerical implementations will perform differently: which formula will be superior in practice is not a priori clear, but will likely depend on the range of parameters involved. For a complete picture, we provide independent proofs of the two given formulas. 
 
 \begin{theorem}\label{thm_p_tilde_2D}  
 Let $\psi(u)$ be the Laplace exponent of $T_1$, and let $\tilde W$ have drift $\beta\ne 0$. Then
 \beq\label{eq_ptilde_2D}
  p_1^*(x_0;s,x_1)&&\nonumber\\&&\hspace{-.75in}=
  \frac{e^{\beta(x_1-x_0)}}{4\pi^2  }\iint\limits_{{\mathbb
 R}^2} \frac{\psi(iz_1)-\psi(iz_2)}{i(z_1-z_2)}\frac{e^{-s\psi(iz_1)}}{\sqrt{\beta^2-2iz_2}}
   e^{-x_0\sqrt{\beta^2-2iz_1}-\mid x_1 \mid\sqrt{\beta^2-2i z_2}}
 dz_1 dz_2\\
&&\hspace{-.75in}=-\frac{2e^{\beta (x_1-x_0)}}{\pi^2}\mbox{PV}\iint_{(\mathbb{R}^+)^2} dk_1dk_2 \frac{k_2\cos|x_1|k_1\sin x_0k_2}{k_1^2-k_2^2}e^{-s\psi((k_1^2+\beta^2)/2)}\psi((k_2^2+\beta^2)/2)\nonumber\\
&&\hspace{-.5in}- \frac{e^{\beta (x_1-x_0)}}{\pi}\int_{\mathbb{R}^+}\ \sin|x_1|k\ \sin x_0k\ e^{-s\psi((k^2+\beta^2)/2)}\psi((k^2+\beta^2)/2) dk
\label{realform} \eeq
Here $PV$ denotes that the principal value contour is taken.
 \end{theorem}
 
 \begin{remark} The equivalence of these two formulas can be demonstrated directly by performing a change of variables $k_j=i\sqrt{\beta^2-2iz_j^2}, \ j=1,2$, followed by a deformation of the contours. Justification of the contour deformation (from the branch of a left-right symmetric  hyperbola in the upper half $k_j$-plane to the real axis) depends on the decay of the integrand, and the computation of certain residues. 
\end{remark}
 
 \subsection{First proof of Theorem \ref{thm_p_tilde_2D}}
 For a fixed level $h>0$, the first passage time and the overshoot
 of the process $T_t$ above the level $h$ are defined to be  
 $\tilde t(h)=\inf\{t>0 \mid T_t>h \}$ and $\tilde \delta(h)=T_{\tilde t(h)}-h$.  
 The Pecherskii-Rogozin identity \cite{PercRogo69} applied to the nondecreasing process $T$ says that
 \beqq
   \int\limits_0^{\infty} e^{-z_1 h} E\left[e^{-z_2 \tilde \delta(h)-z_3 \tilde t(h)}
   \right]dh= \frac{\psi(z_1)-\psi(z_2)}{z_1-z_2}\left(z_3 + \psi(z_1)\right)^{-1},
 \eeqq
Inversion of the Laplace transform in the above equation then leads to
 \beq\label{laplace1}
 E\left[e^{-z_2 \tilde \delta(h)-z_3 \tilde t(h)}
   \right]=\frac1{2\pi} \int\limits_{\mathbb{R}} 
   \frac{\psi(iz_1)-\psi(z_2)}{iz_1-z_2}\left(z_3 + \psi(iz_1)\right)^{-1}\; e^{iz_1 h}\; dz_1
 \eeq
 
The first passage time of the BM with drift is defined as
 $T^*=T^*(x_0)=\inf\{ T>0 \mid x_0+W_T+\beta T<0  \}$. Next, we need to find the joint Laplace 
 transform of $t^*_1=\inf\{ t \mid T_t>T^*\}=\tilde t(T^*)$ and the overshoot $\delta^*=\tilde \delta(T^*)$.
 Since $T_t$ is independent of $W_T$ we find that 
  \beq
  E\left[e^{-z_2 \delta^* -z_3 t^*_1}
   \right]&=&E\left [ E\left[e^{-z_2 \tilde \delta(T^*) -z_3 \tilde t(T^*)} \mid T^* 
   \right] \right]\\ \nonumber
   &=&\frac1{2\pi} \int\limits_{\mathbb{R}} 
   \frac{\psi(iz_1)-\psi(z_2)}{iz_1-z_2}\left(z_3 + \psi(iz_1)\right)^{-1}E\left[e^{iz_1 T^*}\right] dz_1\\ \nonumber
   &=&\frac1{2\pi} \int\limits_{\mathbb{R}} 
   \frac{\psi(iz_1)-\psi(z_2)}{iz_1-z_2}\left(z_3 + \psi(iz_1)\right)^{-1}e^{- x_0 \left(
 \beta+\sqrt{\beta^2-2i
 z_1}\right)}dz_1,
  \eeq
where in the last equality we have used the following well-known result for the characteristic function of the first passage time of BM with drift: 
  \beqq
 E\left[ e^{i z_1 T^*(x_0)}\right]=e^{- x_0 \left(
 \beta+\sqrt{\beta^2-2i
 z_1}\right)}.
 \eeqq

 Next we use the Fourier transform of the PDF of the BM with drift
 in time variable to obtain
 \beq
 E[\delta(\tilde W_t-x_1)]=\frac{e^{-\frac{(x_1-\beta t)^2}{2 t}}}{\sqrt{2\pi t}}=
 \frac{e^{\beta x_1}}{2\pi }
 \int\limits_{{\mathbb R}} e^{-i z_2 t} \frac{e^{-\mid
x_1 \mid \sqrt{\beta^2-2iz_2}}}{\sqrt{\beta^2-2iz_2}}dz_2.
 \eeq
 Thus, using the fact that $\tilde W$ is independent of $t^*_1$ and $\delta^*$  we obtain
 \beqq
 &&E\left[e^{-z_3 t^*_1}\delta\left(\tilde W_{\delta^*}-x_1 \right)
   \right]=E\left[E\left[e^{-z_3 t^*_1}\delta\left(\tilde W_{\delta^*}-x_1 \right)
   \mid \delta^*\right] \right]\\
   &=&\frac{e^{\beta x_1}}{2\pi }
 \int\limits_{{\mathbb R}} E\left[e^{-z_3 t^*_1-i z_2 \delta^*}\right] \frac{e^{-\mid
x_1 \mid \sqrt{\beta^2-2iz_2}}}{\sqrt{\beta^2-2iz_2}}dz_2\\
&=&\frac{e^{\beta(x_1-x_0)}}{4\pi^2  }\iint\limits_{{\mathbb
 R}^2} \frac{\psi(iz_1)-\psi(iz_2)}{i(z_1-z_2)}\frac{\left(z_3 + \psi(iz_1)\right)^{-1}}{\sqrt{\beta^2-2iz_2}}
   e^{-x_0\sqrt{\beta^2-2iz_1}-\mid x_1 \mid\sqrt{\beta^2-2i z_2}}
 dz_1 dz_2.
 \eeqq
 Now, the statement of the Theorem follows after one additional Fourier inversion:
 \beqq
 &&p^*_1(x_0;s,x_1)=E\left[\delta\left(t^*_1-s\right) \left(\tilde W_{\delta^*}-x_1\right) \right]\\
 &=&\frac{1}{2\pi} \int\limits_{{\mathbb R}} e^{iz_3 s} E\left[e^{-i z_3 t^*_1}\delta\left(\tilde W_{\delta^*}-x_1 \right)
   \right] dz_3 \\ 
 &=&\frac{e^{\beta(x_1-x_0)}}{4\pi^2  }\iint\limits_{{\mathbb
 R}^2} \frac{\psi(iz_1)-\psi(iz_2)}{i(z_1-z_2)}\frac{e^{-s\psi(iz_1)}}{\sqrt{\beta^2-2iz_2}}
   e^{-x_0\sqrt{\beta^2-2iz_1}-\mid x_1 \mid\sqrt{\beta^2-2i z_2}}
 dz_1 dz_2,
 \eeqq
 where we have also used the following Fourier integral:
 \beqq
  \frac{1}{2\pi}\int\limits_{{\mathbb R}} \frac{e^{iz_3 s}}{iz_3+\psi(i z_1)} dz_3=e^{-s \psi(i z_1)}. 
 \eeqq
 \qed
 
  \subsection{Second proof of Theorem \ref{thm_p_tilde_2D}}

 The strategy of the proof is to compute
 \be I(u)=E_{0,x_0}[{\bf 1}_{\{s< t^*_1\le s+u\}}\delta(X_{s+u}-x_1)]\ee
 and then take the limit of $I(u)/u$ as $u\to 0+$. 
The key idea is to note that $X_{s+u}=X_{s-}+\tilde W'_{T'_u}$ where $\tilde W', T'$ are copies of $\tilde W, T$, independent of the filtration $\CF_{s-}$. We can then perform the above expectation via an intermediate conditioning on $\CF_{s-}$:  
\beq
E[\delta(X_{s+u}-x_1){\bf 1}_{\{s< t^*_1\le s+u\}}|\CF_{s-}]&&\\
&&\hspace{-2in}={\bf 1}_{\{s< t^*_1\}}E[\delta(\ell+\tilde W'_{T_u}-x_1){\bf 1}_{\{u\ge {t'}^*_1\}}|X_{s-}=\ell].
\eeq

To evaluate the expectations that arise,  we will need the second and third of the following results that were stated and proved in \cite{Hurd07a}:
\begin{lemma}\label{phi2lemma}
\begin{enumerate}
  \item For any $s>0$
  \be E_{0,x}[{\bf 1}_{\{s<t^*_1\}}\delta(X_{s}-y)] ={\bf 1}_{\{y>0\}}\frac{e^{\beta(y-x)}}{2\pi }\int_{\mathbb{R}}\left[e^{iz(x-y)}-e^{iz(x+y)}\right]e^{-s\psi((z^2+\beta^2)/2)}dz.
  \ee  
  \item For any $s>0$ and $\epsilon\in\mathbb{R}$
  \be \label{jointPDF}E_{0,x}[{\bf 1}_{\{s\ge t^*_1\}}\delta(X_{s}-y)] =\frac{e^{\beta(y-x)}}{2\pi }\int_{\mathbb{R}+i\epsilon}e^{iz(x+|y|)}e^{-s\psi((z^2+\beta^2)/2)}dz.
  \ee
  \item For any $k$ in the upper half plane, 
  \be E_{0,x}[{\bf 1}_{\{s<t^*_1\}}e^{-\beta X_{s}+ikX_s}] =\frac{e^{-\beta x}}{2\pi }\int_{\mathbb{R}}\left[\frac{i}{k-z}-\frac{i}{k+z}\right]e^{izx}e^{-s\psi((z^2+\beta^2)/2)}dz.
  \label{phi2formula}
  \ee 
\end{enumerate}
\end{lemma}

First, using \eqref{jointPDF}, we find 
\beq
E[\delta(X_{s+u}-x_1){\bf 1}_{\{s< t^*_1\le s+u\}}|X_{s-}=\ell]&&\\
&&\hspace{-2in}={\bf 1}_{\{s< t^*_1\}}\frac{e^{\beta(x_1-\ell)}}{2\pi}\int_{\mathbb{R}+i\epsilon} dk e^{ik(\ell+|x_1|)}e^{-u\psi((k^2+\beta^2)/2)}.
\eeq
When we paste this expression into the final expectation over  $X_{s-}$ we can use Fubini to interchange
the expection and integral providing we choose $\epsilon>0$. Then we find
\be  
I=\frac{e^{\beta x_1}}{2\pi}\int_{\mathbb{R}+i\epsilon} dk e^{ik|x_1|}e^{-u\psi((k^2+\beta^2)/2)} E_{0,x_0}[e^{ikX_s-\beta X_s}{\bf 1}_{\{s< t^*_1\}}].
\ee
We can now use \eqref{phi2formula} from Lemma \ref{phi2lemma} obtain
\be
I=\frac{e^{\beta (x_1-x_0)}}{(2\pi)^2}\iint_{(\mathbb{R}+i\epsilon)\times\mathbb{R}}  e^{ik|x_1|+izx_0}\left[\frac{i}{k-z}-\frac{i}{k+z}\right]e^{-u\psi((k^2+\beta^2)/2)-s\psi((z^2+\beta^2)/2)}dz dk.
\ee
Noting that $I(0)=0$ and taking $\lim_{u\to 0}I(u)/u$ now gives 
\beq
p^*_1(x_0;s,x_1)&=& \\ \nonumber
&&\hspace{-1.2in}\frac{e^{\beta (x_1-x_0)}}{2\pi^2}\iint_{(\mathbb{R}+i\epsilon)\times\mathbb{R}} dkdz \frac{iz}{k^2-z^2}e^{ik|x_1|+izx_0}e^{-s\psi((z^2+\beta^2)/2)}\psi((k^2+\beta^2)/2) \label{complexform}. \eeq
Here the arbitrary parameter $\epsilon>0$ can be seen to ensure the correct prescription for dealing with the pole at $k^2=z^2$.

Finally, the complex integration 
 in \eqref{complexform} can be expressed in the following manifestly real form:
 \beq
 p^*_1(x_0;s,x_1)&=&-\frac{2e^{\beta (x_1-x_0)}}{\pi^2}\mbox{PV}\iint_{(\mathbb{R}^+)^2} dkdz \frac{z\cos|x_1|k\sin x_0z}{k^2-z^2}e^{-s\psi((k^2+\beta^2)/2)}\psi((z^2+\beta^2)/2)\nonumber\\
&&- \frac{e^{\beta (x_1-x_0)}}{\pi}\int_{\mathbb{R}^+}\ \sin|x_1|z\ \sin x_0z\ e^{-s\psi((z^2+\beta^2)/2)}\psi((z^2+\beta^2)/2) dz.
\label{realform2} \eeq
  involving a principle value integral  plus  explicit  half residue terms for the poles $k=\pm z$.
  \qed

  \section{The iteration scheme and its convergence} 
The next theorem shows that \eqref{p_iteration}  can be used to compute $p^*(x_0;s,x)$. We define a suitable $L^\infty$ norm for functions $f(x_0;u,x)$:
\be   
\|f\|_\infty=\sup\limits_{x_0 \ge 0} \left[\int\limits_0^{\infty} \int\limits_{0}^{\infty} \;
    |f(x_0;u,x)|du dx\right].
    \ee

     \begin{theorem}\label{thm_exp_convgce}
     The sequence $(p_n^*)_{n\ge 1}$  converges exponentially in the $L^{\infty}$ norm.
     \end{theorem} 
     \begin{proof}
     First we find from (\ref{p_iteration})
     \beqq
     p_{n+1}^*(x_0;s,x_1)-p_{n}^*(x_0;s,x_1)=\int\limits_0^{\infty} dy \int\limits_{0}^s du \;
    p_1^*(x_0;s-u,y)\left[p_n^*(y;u,x_1) -p_{n-1}^*(y;u,x_1) \right],
     \eeqq
     thus 
     \beqq
     \| p_{n+1}^*-p_{n}^* \|_{\infty}\le C \| p_{n}^*-p_{n-1}^* \| _{\infty}
     \eeqq
     where 
     \beqq
      C=\sup\limits_{x_0 \ge 0} \left[\int\limits_0^{\infty} \int\limits_{0}^{\infty} \;
    p_1^*(x_0;u,x)du dx\right].
     \eeqq
     
     The proof is based on the probabilistic interpretation of the constant $C$: by definition $p_1^*(x_0;u,x)$
     is the joint density of $t^*_1$ and $X_{t^*_1}$, thus we obtain: 
     \beqq
     C=\sup\limits_{x_0 \ge 0} P\left(t_1^*<+\infty,\;  X_{t_1^*}>0  \mid X_0=x_0\right).
     \eeqq
     Next, using the fact that $\tilde W_{T^*}=0$ ($T^*$ is the first passage time of $X_T$ and $X$ is a continuous process)
     and the strong Markov property of the Brownian motion we find that 
     \beqq
     C
     &=&
     P\left(t_1^*<+\infty,\;  \tilde W_{T_{t_1^*}}-\tilde W_{T^*}>0  \mid \tilde W_0=x_0\right)\\
     &=& P\left(t_1^*<+\infty,\;  W_{\delta^*}+\beta\delta^*>0  \mid W_0=0\right),
     \eeqq
     where the Brownian motion $W_t$ is independent of $T_t$ and $\delta^*=\delta^*(x_0)=T_{t_1^*}-T^*$ 
     is the overshoot of the time change above $T^*$. 
     
     Thus we need to prove that 
     \beq
     C=\sup\limits_{x_0 \ge 0} P\left(t_1^*<+\infty,\;  W_{\delta^*}+\beta\delta^*>0  \mid W_0=0\right) <1,
     \eeq
     where $t_1^*=t_1^*(x_0)$, $\delta^*=\delta^*(x_0)$ and the Brownian motion $W$ is independent of $t_1^*$ and $\delta^*$.

     First we will consider the case when $\beta<0$. In this case we  obtain
     \beqq
     P\left(t_1^*<+\infty,\;  W_{\Delta^*}+\beta\delta^*>0  \mid W_0=0\right)&\le&
     P\left( W_{\delta^*}+\beta\delta^*>0  \mid W_0=0\right)\\
     &&\hspace{-6cm}=\ \int\limits_{0}^{\infty} P( W_t+\beta t>0 \mid W_0=0 )P(\delta^* \in dt)<
     \int\limits_{0}^{\infty} \frac12 P(\delta^* \in dt)=\frac12,
     \eeqq
     where we have used the fact that $W_t$ is independent of the overshoot  $\delta^*$ and that 
     $P( W_t+\beta t>0 \mid W_0=0 )<\frac12$ for any $t$ and any $\beta<0$. Thus in the case when the drift $\beta$ is negative we 
     obtain an estimate $C<\frac12$.

     The case when the drift $\beta$ is positive is more complicated. 
     We can not use the same techniques as before, since the bound
     $P( W_t+\beta t>0 \mid W_0=0 )<\frac12$ is no longer true: in fact    
     $P( W_t+\beta t>0 \mid W_0=0 )$ monotonically increases to $1$ as $t\to \infty$.

     First we will consider the case when $x_0$ is bounded away from $0$: $x_0\ge c>0$. 
     Then $x_0+W_t+\beta t$ has a positive probability of escaping to $+\infty$ and never crossing the barrier at 0, 
     thus 
     \beqq
      P\left(t_1^*<+\infty,\;  W_{\delta^*}+\beta\delta^*>0  \mid W_0=0\right)&\le& P\left(t_1^*(x_0)<+\infty\right)\\
     &<&P\left(t_1^*(c)<+\infty\right)=1-\epsilon_1(c).
     \eeqq

     Now we need to consider the case when $x_0 \to 0^+$. 
     The proof in this case is based on the following sequence of inequalities:
     \beq\label{big_inequalities}
     P\left(t_1^*<+\infty,\;  W_{\delta^*}+\beta\delta^*>0  \mid W_0=0\right)\hspace{-6cm}&& \\ \nonumber
     &\le& P\left( W_{\delta^*}+\beta\delta^*>0  \mid W_0=0\right)\\\nonumber
     &=&
     1-P\left( W_{\delta^*}+\beta\delta^*<0  \mid W_0=0\right)\\\nonumber
     &=&1-\int\limits_{0}^{\infty} P( W_t+\beta t<0 \mid W_0=0 )P(\delta^* \in dt)\\\nonumber
     &<& 1-\int\limits_{0}^{\tau} P( W_t+\beta t<0 \mid W_0=0 )P(\delta^* \in dt)\\\nonumber
     &<& 1-\int\limits_{0}^{\tau} P( W_{\tau}+\beta {\tau}<0 \mid W_0=0 )P(\delta^* \in dt)\\\nonumber
     &=&1- P( W_{\tau}+\beta {\tau}<0 \mid W_0=0 )P(\delta^* < \tau),
     \eeq
     where $\tau$ is any positive number and the last inequality is true since $P( W_t+\beta t<0 \mid W_0=0 )$
     is a decreasing function of $t$. 
     
     Since $x_0 \to 0^+$, we also have $T^*(x_0) \to 0^+$ with probability 1. Since $\delta^*$ is the overshoot of $T^*$,  
     and $T^* \to 0^+$ as $x_0 \to 0^+$, we see that the 
      distribution of the overshoot $\delta^*(x_0)$
     converges either to the distribution of the jumps of $T_t$ if the time change process $T_t$ is a compound Poisson
      process or to the Dirac delta distribution at $0$ if $T_t$ has infinite activity of jumps. Therefore in the case 
      when $T_t$ is a compound Poisson process
      with the jump measure $\nu(dx)$ we choose $\tau$ such that $\nu([0,\tau])>0$, and if $T_t$ has infinite activity of jumps 
      we can take any $\tau>0$. Then we obtain $\lim\limits_{x_0\to 0^+} P(\delta^*(x_0)<\tau)=\xi$, where
      $\xi=\nu([0,\tau])$ in the case 
      of compound Poisson process, and $\xi=1$ in the case of the process with infinite activity of 
      jumps. 
      Using (\ref{big_inequalities}) we find that as $x_0\to 0^+$
      \beqq
      P\left(t_1^*(x_0)<+\infty,\;  W_{\delta^*}+\beta\delta^*>0  \mid W_0=0\right)
      &<&1-P( W_{\tau}+\beta {\tau}<0 \mid W_0=0 )\xi\\
      &<&1-\epsilon_2.
      \eeqq
      
      To summarize, we have proved that the function 
      \beqq
      P(x_0)=P\left(t_1^*(x_0)<+\infty,\;  W_{\delta^*}+\beta\delta^*>0  \mid W_0=0\right),
      \eeqq
      satisfies the following properties: 
      \begin{itemize}
      \item for any $c>0$ there exists $\epsilon_1=\epsilon_1(c)>0$ such that $P(x_0)<1-\epsilon_1(c)$ for all $x_0>c$ 
      \item  there exists $\epsilon_2>0$ such that $\lim\limits_{x_0\to 0^+} P(x_0)< 1-\epsilon_2$ .
      \end{itemize}
       Therefore we conclude that there exists 
      $\epsilon>0$, such that  $P(x_0)<1-\epsilon$ for 
      all $x_0\ge 0$, thus $C<1-\epsilon$. This ends the proof in the second case $\beta >0$. 
     \end{proof}

For a complementary point of view, the next result shows that the sequence $(t^*_i(x_0)_{i\ge 1}$  converges pathwise. 
\begin{theorem} For any TCBM with L\'evy subordinator $T_t$ and Brownian motion with drift $\beta$ the sequence of stopping times $(t^*_i(x_0)_{i\ge 1}$  converges a.s to $t^*$.
  \end{theorem}

\noindent{\bf Proof:\ }\footnote{The authors would like to thank Professor Martin Barlow for providing ths proof.} If $t^*=\infty$, then certainly $t^*_i\to\infty$, so we suppose $t^*<\infty$. In this case, if $t^*_i=t^*_{i+1}$ for some $i$ the sequence converges, and thus the only interesting case to analyze is if $t^*\ne t^*_i$ for all $ i<\infty$. Then we have $t^*_1<t^*_2<\dots<t^*_i<\dots$. Correspondingly, we have an  infinite sequence of excursion overjump intervals which do not overlap: let their endpoints be $T^*_{i-}:=T_{t^*_i-}<T^*_{i+}:=T_{t^*_i}$. The following observations lead to the conclusion:  
\begin{enumerate}
\item by monotonicity and boundness of the sequences $(T^*_{i-})$ and $(T^*_{i+})$, $\lim_{i\to\infty}T^*_{i-}=\lim_{n\to\infty}T^*_{i+}=T_\infty$ exists;
\item $x_0+W_{T_\infty}+\beta T_\infty=0$ by the continuity of Brownian motion;
\item $\lim_{i\to\infty}t^*_i=t_\infty$ exists, and $t_\infty\le t^*$;
\item Jump times are totally inaccessible, so there is no time jump at time $t_\infty$ almost surely, hence $T_{t_\infty}=T_\infty$;
\item Therefore $X_{t_\infty}=0$ and so $t_\infty\ge t^*$: hence $t_\infty= t^*$.
\end{enumerate} 
\qed

 \section{The Variance Gamma model}

The Variance Gamma (VG) process described in Section 2 is the LSBM where the time change process $T_t$ is the L\'evy process with jump measure $\mu(z)=(\nu z)^{-1}\exp(-z/\nu)$ on $(0,\infty)$ and the Laplace exponent $\psi_T(u)=\frac{1}{\nu}\log(1+\nu u )$. In this section we take the parameter $b=0$. This model has been widely used for option pricing where it has been found to provide a better fit to market data than the Black-Scholes model, while preserving a degree of analytical tractability. The main result in this section reduces the 2D integral representation (\ref{eq_ptilde_2D}) for $p_1^*(x_0;s,x_1)$ to a 1D integral and leads to greatly simplified numerical computations. 
  \begin{theorem}
   Define 
 $\alpha=\sqrt{\frac{2}{\nu}+\beta^2}$. Then 
 \beq\label{eq_ptilde1}
p_1^*(x_0;s,x_1)&=&\frac{e^{\beta(x_1-x_0)}}{2\pi \nu
 }\int\limits_{{\mathbb
 R}} \frac{(1+i\nu
 z)^{-\frac{s}{\nu}}}{\sqrt{\beta^2-2iz}}
   e^{-x_0\sqrt{\beta^2-2iz}}
   \\ \nonumber && \times
   \bigg[e^{\mid x_1 \mid\sqrt{\beta^2-2iz}}Ei\left(-\mid x_1 \mid\left(\alpha+ \sqrt{\beta^2-2iz}\right) \right)- \\ \nonumber
   && \;\;\;\;\;\; e^{-\mid x_1 \mid\sqrt{\beta^2-2iz}}Ei\left(-\mid x_1 \mid\left(\alpha- \sqrt{\beta^2-2iz}\right) \right)\bigg]
 dz,
 \eeq
where $Ei(x)$ is the exponential integral function (see \cite{GradRyzh00}).
  \end{theorem}
  \begin{proof}
  Consider the function $I(z_1)$ which represents the outer integral in (\ref{eq_ptilde_2D}) 
  \beqq
  I(z_1)=\frac{1}{2\pi}\int\limits_{{\mathbb R}} 
  \frac{\log(1+i\nu z_1)-\log(1+i\nu z_2)}{i(z_1-z_2)}\frac{ e^{-\mid x_1 \mid\sqrt{\beta^2-2i z_2}}}{\sqrt{\beta^2-2iz_2}} dz_2
  \eeqq
  First we perform the change of variables $u=i\sqrt{\beta^2-2iz_2}$ and obtain 
  \beq\label{eq_I}
   I(z_1)
   =\frac{1}{\pi} \int\limits_{\mathbb R} \frac{\log\left(1+\frac{\nu}{2}(u^2+\beta^2) \right)-\log(1+i\nu z_1)}{u^2+\beta^2-2iz_1}
    e^{i \mid x_1 \mid u}du,
  \eeq
  where the contour $L$ obtained from ${\mathbb R}$ under map $z_2\to u=i\sqrt{\beta^2-2iz_2}$ is transformed into 
  contour ${\mathbb R}$ (this is justified since the integrand is an analytic function in this region for any $z_1$).
  To finish the proof we separate the logarithms 
  \beqq
  \log\left(1+\frac{\nu}{2}(u^2+\beta^2) \right)-\log(1+i\nu z_1)=\log(u+i\alpha)+\log(u-i\alpha)-\log\left(\frac{2}{\nu}+2iz_1\right),
  \eeqq
  use the partial fractions decomposition
  \beqq
  \frac{1}{u^2+\beta^2-2iz_1}=\frac{1}{2i\sqrt{\beta^2-2iz_1}}\left[\frac{1}{u-i\sqrt{\beta^2-2iz_1}}-\frac{1}{u+i\sqrt{\beta^2-2iz_1}} \right]
  \eeqq 
  and we obtain the six integrals, which can be computed by shifting the contours of integration 
   and using the following  Fourier transform formulas (see Gradshteyn-Ryjik...)
   \beqq
   \int\limits_{i\epsilon+{\mathbb R}} \log\left(1+\frac{iy}{b}\right)e^{ixy}\frac{dy}{y}&=&-2\pi i Ei(-bx), \;\;\; b>0  \\
   \int\limits_{i(b+\epsilon)+{\mathbb R}} \log\left(\frac{iy}{b}-1\right)e^{ixy}\frac{dy}{y}&=&-2\pi i Ei(bx), \;\;\; b>0
   \eeqq
  \end{proof}

    {\bf Remark:} Using the change of variables $u=i\sqrt{\beta^2-2iz}$ and 
simplifying the expression we can obtain a simpler formula
  for $p_1^*(x_0;s,x_1)$:
    \beqq
p_1^*(x_0;s,x_1)&=&e^{\beta(x_1-x_0)}\frac{\left(\frac{\nu}2\right)^{-\frac{s}{\nu}-1}}{2\pi
i }\int\limits_{{\mathbb R}}
\left(\alpha^2+y^2\right)^{-\frac{s}{\nu}}
\sin\left(x_0y \right)  e^{i\mid
x_1\mid y} Ei\left(-\mid x_1
\mid (\alpha+iy) \right)
  dy.
  \eeqq
Applying the Plancherel formula to the above expression gives us the following representation for $p_1^*$: 
 \beq\label{eq_ptilde2}
p_1^*(x_0;s,x_1)&=&e^{\beta(x_1-x_0)-\alpha(x_0+\mid
x_1 \mid)}\frac{\left(\frac{\nu \alpha^2}{2}
\right)^{-\frac{s}{\nu}}}{\sqrt{\pi}\nu(x_0+\mid x_1
\mid)}\frac{\Gamma\left(\frac{s}{\nu}+\frac12\right)}{\Gamma\left(\frac{s}{\nu}+1\right)}\\
\nonumber
&+&\sqrt{\frac{2\alpha^3}{\pi}}e^{\beta(x_1-x_0)-\alpha\mid
x_1 \mid } \frac{\left(\alpha \nu
\right)^{-\frac{s}{\nu}-1}}{\Gamma\left(\frac{s}{\nu} \right) }
\int\limits_{0}^{\infty} u^{\frac{s}{\nu}-\frac12}
K_{\frac{s}{\nu}-\frac12} \left(\alpha u\right) f(x_0,x_1;u)
  du,
 \eeq
 where
 \beqq
 f(x_0,x_1;u)=\frac{e^{-\alpha
\left(u+x_0\right)}}{u+x_0+\mid x_1
\mid}-{\textrm{sign}}\left(u-x_0\right)
\frac{e^{-\alpha {\big\vert
u-x_0\big\vert}}}{\big\vert 
u-x_0\big\vert+\mid x_1
\mid}-2\frac{e^{-\alpha\left(u+x_0\right)}}{x_0+\mid
x_1 \mid}.
 \eeqq
The above expression is useful for computations when $s$ is small. In particular, when $s=0$ we find 
\beq\label{pstars0}
p_1^*(x_0;0,x_1)=\frac{e^{\beta(x_1-x_0)-\alpha(x_0+\mid
x_1 \mid)}}{\nu(x_0+\mid x_1
\mid)}.
\eeq

 \section{Numerical implementation for VG model}

The algorithm for computing the functions $p^*(x_0;s,x)$ and $p^*(x_0;s)$ can be summarized as follows:
 \begin{enumerate}
  \item Choose the discretization step sizes $\delta_x$, $\delta_t$ and discretization intervals $[-X,X]$, $[0,T]$. 
 The grid points are $t_i=i\delta_t$, $1 \le i \le N_s$ and $x_j = (j+1/2) \delta_x$, $-N_x\le j \le N_x$.
 \item Compute the 3D array $p^*_1(x_i; t_j, x_k)$. For $j>0$ use equation (\ref{eq_ptilde1}) and for $j=0$ use explicit formula (\ref{pstars0}).
 \item Iterate equation (\ref{p_iteration}) or \eqref{pdf_iteration}. This step can be considerably accelerated if the convolution in $u$-variable is done using Fast Fourier Transform methods.  We used the midpoint rule for integration in the $y$ and $u$ variables.
 \end{enumerate}

Theorem \ref{thm_exp_convgce} implies that  Step 3 in the above algorithm has to be repeated only a few times: In practice we found that 3-4 iterations is usually enough.
 An important empirical fact is that the above algorithm works quite well with just a few discretization points  in the $x$-variable. We found that if one uses  a non-linear grid (which places more 
 points $x_i$ near $x=0$) then the above algorithm produces reasonable results with values of $N_x$ as small as  10 or 20.

  \begin{figure}\label{fig2}
\begin{center}
 \includegraphics[width=0.45\textwidth]{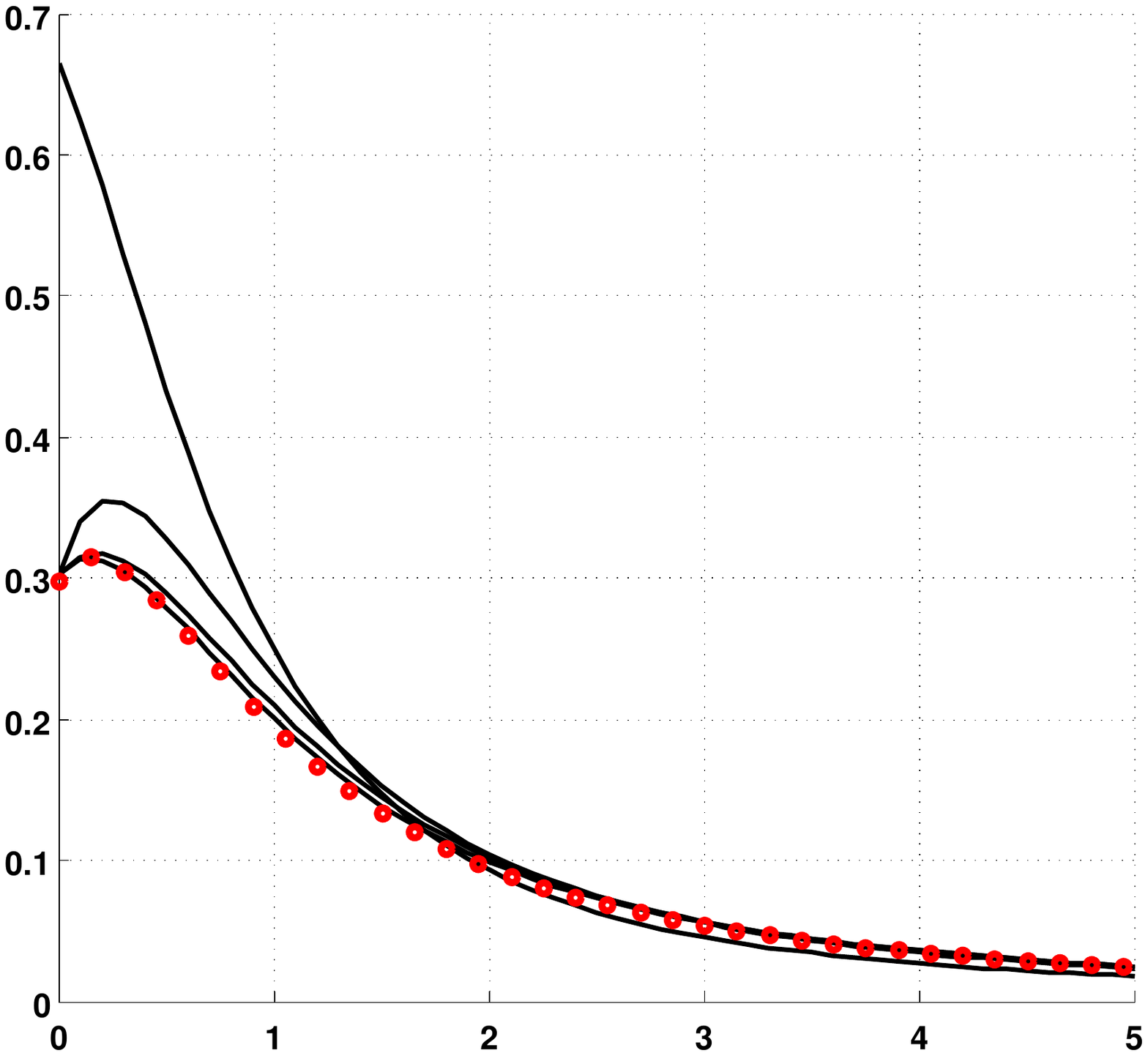}
\includegraphics[width=0.45\textwidth]{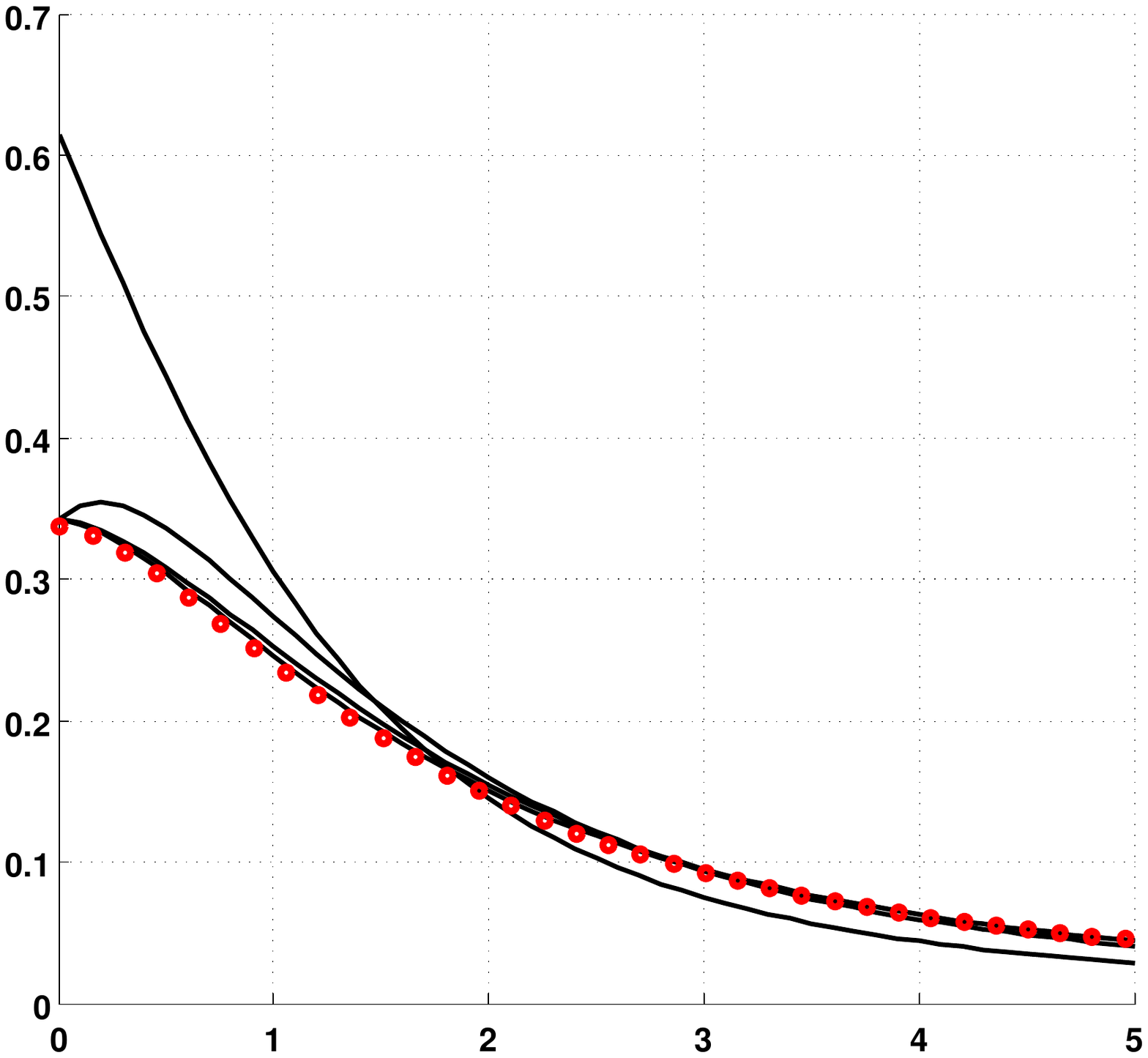}

\vspace{0.5cm} \caption{The density of the first passage time for the two set of parameters. The circles show the ``exact'' result; the three solid lines show the first three approximations.}
\end{center}
\end{figure}

We compared our algorithm for the PDF $p^*(x_0;s)$ to a finite-difference method that was implemented as follows. First we approximated the first passage time by its discrete counterpart:
 \beqq
 \hat t^*=\hat t^*(x)=\min\{t_i : X_{t_i}<0 | X_0=x \}
\eeqq  
where $t_i=i\delta_t$, $0\le i \le n_t$ is the discretization of the interval $[0,T]$. The probabilities $f_i(x)=P(\hat t^* > t_i | X_0=x)$ satisfy the iteration:
 \be\label{pide_iteration}
  f_{i+1}(x)={\bf 1}_{x>0}\int \limits_{{\mathbb R}} p(\delta_t, x-y) f_i(y)dy,  i\ge 1
 \ee
with $f_{0}(x)={\bf 1}_{x>0}$ and can be computed numerically with the following steps:
\begin{enumerate}
 \item discretize the space variables $x=i \delta_x$, $y=j \delta_x$, $0<i,j<n_x$;
 \item compute the array of transitional probabilities $\hat p_i=p(\delta_t,x_i)$, and normalize $\hat p_0$ so that $\sum_i \hat p_i=1$; 
 \item use the convolution (based on FFT) to iterate equation (\ref{pide_iteration}) $n_t$ times;
 \item compute the approximation of the first passage time density $\hat p^*(x,t_i+\delta_t/2)=(f_{i+1}(x)-f_i(x))/\delta_t$.  
\end{enumerate}

The big advantage of this method is that it is explicit and unconditionally stable: we can choose the number of discretization points in $x$-space and $t$-space independently. This is not true in general explicit finite difference methods, where one would solve the Fokker-Plank equation by discretizing the Markov generator and derivative in time, since $\delta_t$ and $\delta_x$ have to lie in a certain subset in order for the methods to be stable.

Figure \ref{fig2}  summarizes the numerical results for the PDF $p^*(x_0;s)$ over the time interval $[0,5]$ for the VG model with the following two sets of parameters:
 \beqq
&& \mbox{Set I:}\;\;\; x_0=0.5, \;\; \beta = 0.2, \;\; \nu=1 \\
&&\mbox{Set II:} \;\; x_0=0.5, \;\; \beta = -0.2, \;\; \nu=2 .
\eeqq 
 The number of grid points used was $N_t=50$ and $N_x=10$.  The red circles correspond to the solution obtained by a high resolution finite difference PIDE method as described above (with $n_t=1000$ and $n_x=10 000$), and the 
black lines show successive iterations $p^*_i(x_0,t)$ converging to $p^*(x_0,t)$. As we see, 3 iterations of equation (\ref{pdf_iteration}) provide a visually acceptable accuracy in a running time of less than 0.1 second (on a 2.5Ghz laptop).

   \begin{figure}\label{fig3}
\begin{center}
 \includegraphics[width=0.45\textwidth]{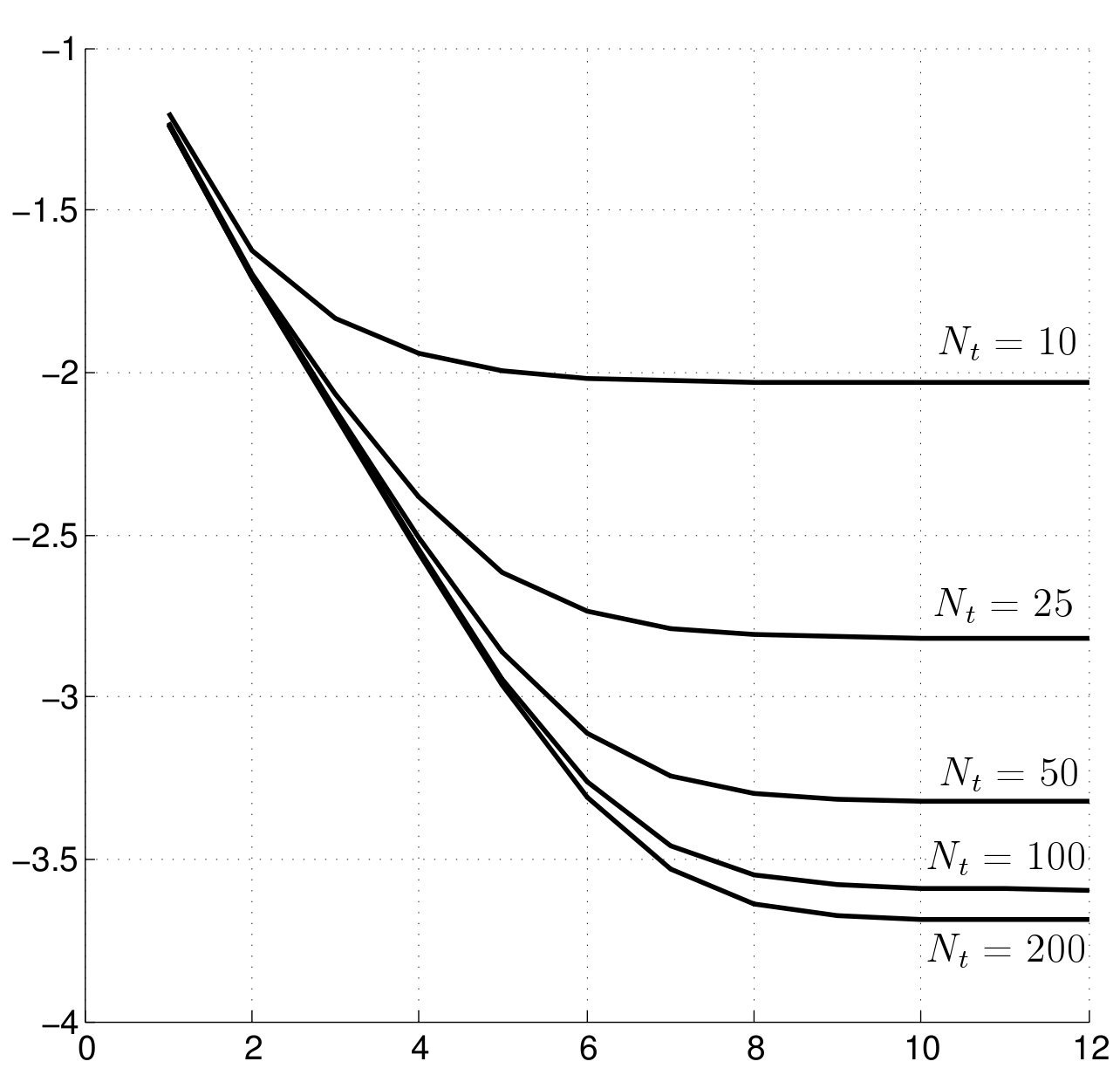}
\includegraphics[width=0.45\textwidth]{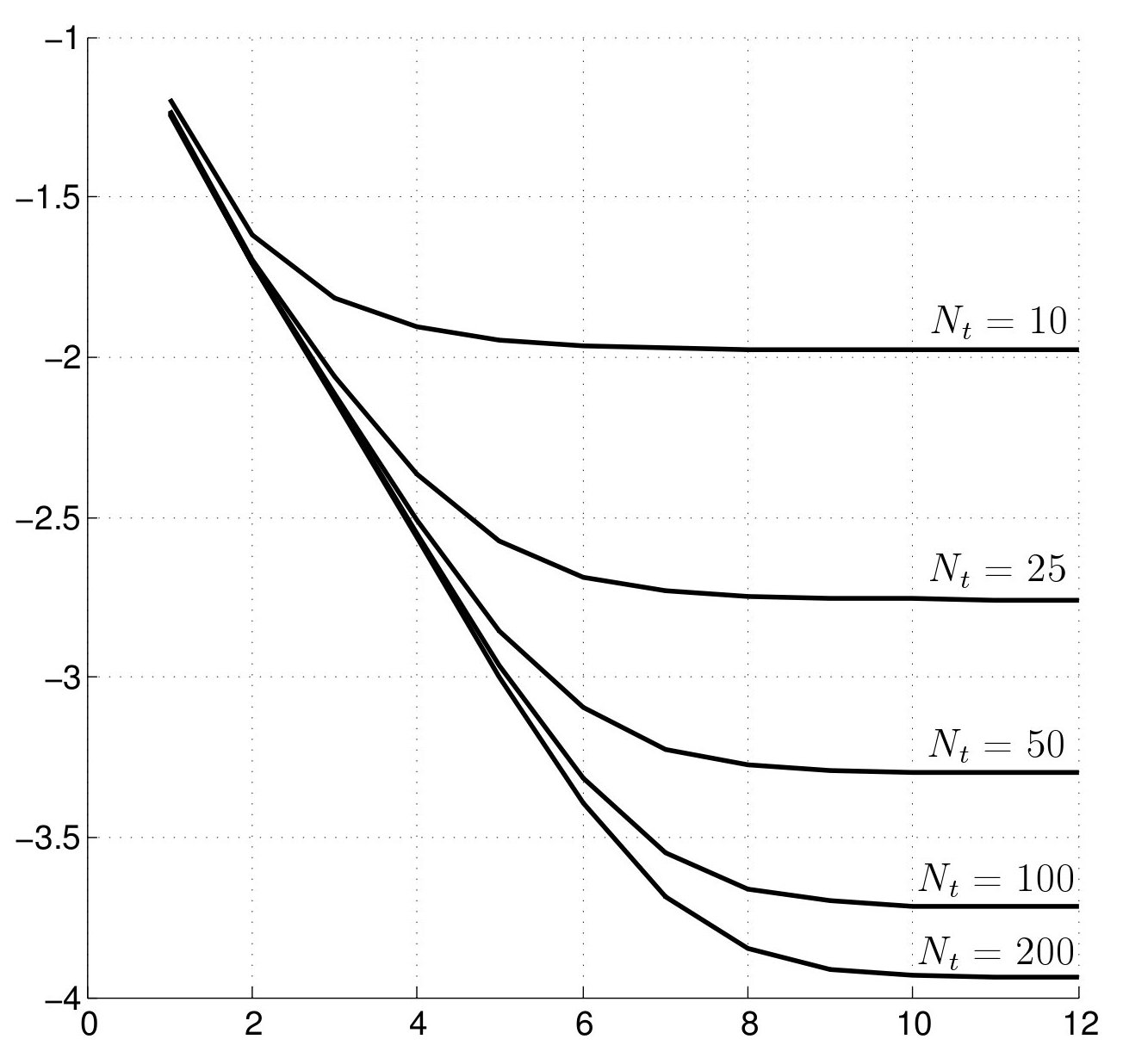}
\vspace{0.5cm} \caption{Vertical axis is the $\log_{10}\left(||p^*-p^*_i||_{L_1}\right)$, horizontal axis is the number of iterations. $N_x=10$ (left) and $N_x=20$ (right) and $N_t \in \{10, 25, 50, 100, 200\}$ }
\end{center}
\end{figure}

 Figure \ref{fig3} illustrates the convergence of our method and Table 1 shows the computation times (on the same 2.5Ghz laptop). We used Set II of parameters for the VG process, and the PIDE method with  $n_t=1000$ and $n_x=10 000$ to compute the ``exact'' solution  $p^*(x_0,t)$. 	Figure \ref{fig3} shows the $\log_{10}$ of the error  
 \beqq
||p^*-p^*_i||_{L_1}=\int\limits_0^T | p^*(x_0,t)-p^*_i(x_0,t) | dt
\eeqq
on the vertical axis and the number of iterations on the horizontal axis; different curves correspond to different number of discretization points in $t$-space. The number of discretization points in $x$-space is fixed at $N_x=10$ for the left picture and $N_x=20$ for the right picture. We see that initially the error decreases exponentially and then flattens out. The flattening indicates that our method converges to the wrong target (which is to be expected since there is always a discretization error coming from $N_x$ and $N_t$ being finite). However, increasing $N_t$ and $N_x$ brings us closer to the ``target''. In the table 1 we show precomputing time needed to compute the 3D array $p^*_1(x_i; t_j, x_k)$ and the time needed to perform each iteration  (\ref{pdf_iteration}).

\begin{table}[!ht]\label{tab1}
\centering
\label{table_comp_times_our_method}
\bigskip
\begin{tabular}{|c|c||c|c|c|c|c|c|}
\hline
 	 			  &     &$N_t=10$  & $N_t=25$ & $N_t=50$ & $N_t=100$ & $N_t=200$  \\ \hline \hline
  {\textrm precomputing time}	  &   $N_x=10$   &0.0313    & 0.0259 & 0.0324 & 0.0461  & 0.0687 \\ \hline
  {\textrm each iteration}	  &   $N_x=10$   &0.0006  & 0.0008 & 0.0011 & 0.0021  & 0.0046 \\ \hline \hline
  {\textrm precomputing time}	  &   $N_x=20$   &0.0645   & 0.0612 & 0.0745 & 0.0868  & 0.1298 \\ \hline
  {\textrm each iteration}	  &   $N_x=20$   &0.0037  & 0.0045 & 0.0066 & 0.0120  & 0.0269 \\ \hline 		
\end{tabular}
\caption{Computation time (sec) for the new approach}
\end{table}

   \begin{figure}\label{fig4}
\begin{center}
 \includegraphics[width=0.55\textwidth]{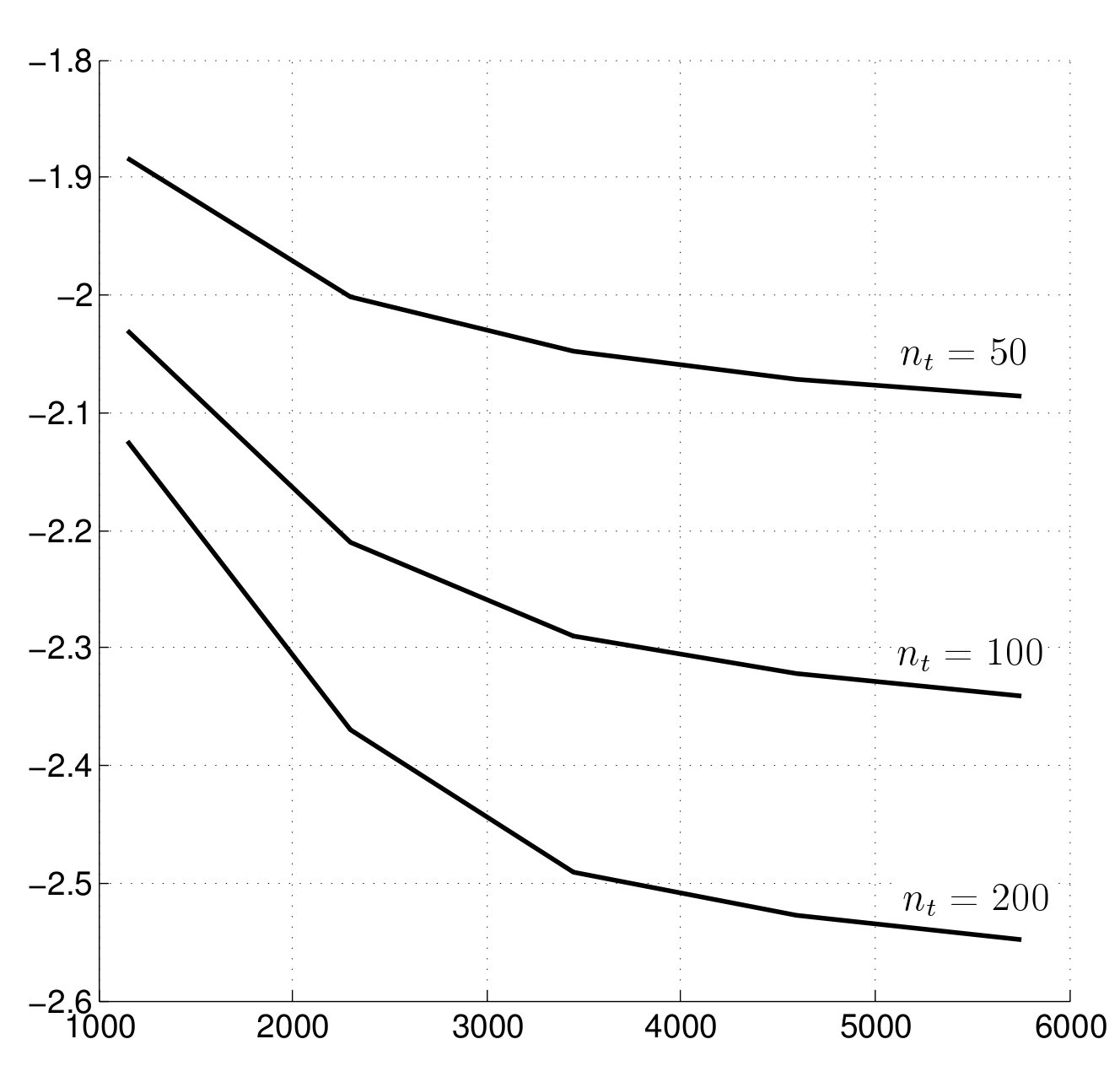}
\vspace{0.5cm} \caption{Vertical axis is the $\log_{10}\left(||p^*-\hat p^*||_{L_1}\right)$, horizontal axis is $n_x$. The $n_t$ is  in the set $\{50, 100, 200\}$}
\end{center}
\end{figure}

To put these results into perspective, on Figure \ref{fig4} and Table 2 we present similar results for finite difference method. On Figure \ref{fig4} we show the same logarithm of the error on the vertical axis, and the number of discretization points $n_x$ on the horizontal axis. Different curves correspond to $n_t \in \{50, 100, 200\}$. The running time presented in Table 2 includes only the time needed to perform $n_t$ convolutions (\ref{pide_iteration}) using the FFT. As we see, the finite difference method is substantially slower than our method.

\begin{table}[!ht]\label{tab2}
\centering
\label{table_comp_times_PIDE_method}
\bigskip
\begin{tabular}{|c||c|c|c|c|c|}
\hline
              &$n_x=1150$  & $n_x=2300$ & $n_x=3450$ & $n_x=4600$ & $n_x=5750$   \\ \hline \hline
   $n_t=50$   &0.0756    & 0.2935 & 0.9582 & 1.8757  & 2.9967 \\ \hline
    $n_t=100$   &0.1456  & 0.5821 & 1.9397 & 3.7409  & 6.0026 \\ \hline 
    $n_t=200$   &0.2870   & 1.1478 & 3.8833 & 7.4768  & 11.9935 \\ \hline		
\end{tabular}
\caption{Computation time (sec) for the finite difference approach.}
\end{table}

 \section{Conclusions}
 First passage times are an important modeling tool in finance and other areas of applied mathematics. 
 The main result of this paper is the theoretical connection between two distinct notions of first passage time that arise for L\'evy subordinated Brownian motions. This relation leads to a new way to compute true first passage for these processes that is apparently less expensive than finite difference methods for a given level of accuracy. Our paper opens up many avenues for further theoretical and numerical work. For example, the methods we describe are certainly applicable for a much broader class of time changed Brownian motions and time changed diffusions. Finally, it will be worthwhile to explore the use of the first passage of the second kind is a modeling alternative to the usual first passage time.   
 
 \bibliographystyle{plain}

\begin{thebibliography}{10}

\bibitem{AlilKypr05}
L.~Alili and A.~E. Kyprianou.
\newblock Some remarks on first passage of {L}\'evy processes.
\newblock {\em Ann. Appl. Probab.}, 15:2062Ð2080, 2005.

\bibitem{Applebaum04}
David Applebaum.
\newblock {\em L\'evy processes and stochastic calculus}, volume~93 of {\em
  Cambridge Studies in Advanced Mathematics}.
\newblock Cambridge University Press, Cambridge, 2004.

\bibitem{AsmAvrPis04}
S.~Asmussen, F.~Avram, and M.R. Pistorius.
\newblock Russian and american put options under exponential phase-type lŽvy
  models.
\newblock {\em Stochastic Processes and their Applications}, 109:79--111, 2004.

\bibitem{Barn-Niel97}
O.~E. Barndorff-Nielsen.
\newblock Normal inverse {G}aussian distribution and stochastic volatility
  modelling.
\newblock {\em Scandinavian Journal of Statistics}, 24:1--13, 1997.

\bibitem{CherShir02}
A.~S. Cherny and A.~N. Shiryaev.
\newblock Change of time and measures for {L}\'evy processes.
\newblock Lectures for the Summer School ``From {L}\'evy Processes to
  Semimartingales: Recent Theoretical Developments and Applications to
  Finance'', Aarhus 2002, 2002.

\bibitem{GradRyzh00}
I.~S. Gradshteyn and I.~M Ryzhik.
\newblock {\em Tables of integrals series and products, 6th edition}.
\newblock Academic Press, 2000.

\bibitem{Hurd07a}
T.~R. Hurd.
\newblock Credit risk modelling using time-changed {B}rownian motion.
\newblock Working paper
  \verb+http://www.math.mcmaster.ca/tom/HurdTCBMRevised.pdf+, 2007.

\bibitem{JacoShir87}
J.~Jacod and A.~N. Shiryaev.
\newblock {\em Limit theorems for stochastic processes}.
\newblock Springer-Verlag, Berlin, 1987.

\bibitem{KouWang03}
S.~G. Kou and H.~Wang.
\newblock First passage times of a jump diffusion process.
\newblock {\em Adv. in Appl. Probab.}, 35(2):504--531, 2003.

\bibitem{MadaSene90}
D.~Madan and E.~Seneta.
\newblock The {VG} model for share market returns.
\newblock {\em Journal of Business}, 63:511--524, 1990.

\bibitem{PercRogo69}
E.~A. Percheskii and B.~A. Rogozin.
\newblock On the joint distribution of random variables associated with
  fluctuations of a process with independent increments.
\newblock {\em Theory Probab. Appl.}, 14:410Ð423, 1969.

\end{thebibliography}

\end{document}